\documentclass{siamltex1213}

\usepackage{amsmath,amsfonts,amssymb,enumerate}

\newcommand{\R}{\mathbb R}
\newcommand{\C}{\mathbb C}
\newcommand{\D}{\mathbb D}
\newcommand{\q}{{\cal Q}}
\newcommand{\qcb}{{\cal Q}_{cb}}
\newcommand{\pc}{\psi_{_{cb}}}

\title{Some constants related to Numerical Ranges}

\author{Michel {\sc Crouzeix} \footnotemark[2]}

\begin{document}
\maketitle
\renewcommand{\thefootnote}{\fnsymbol{footnote}}
 \footnotetext[2]{Institut de Recherche Math\'ematique de Rennes, UMR CNRS n$^\circ$ 6625 \\
Universit\'e de Rennes 1, 
Campus de Beaulieu, 35042 RENNES Cedex, France\\
{\tt michel.crouzeix@univ-rennes1.fr} }
\begin{abstract}
In an attempt to progress towards proving the conjecture {\it the numerical range $W(A)$ is a 2--spectral set for the matrix $A$}, we propose a study of various constants. We review some partial results; many problems are still open. We describe our corresponding numerical tests.

\end{abstract}

\section{Introduction}
 Let us introduce our notation. We denote $\|v\|=(v^*v)^{1/2}$ the usual Euclidean norm of a column vector $v\in \C^d$;
$\|M\|:=\sup\{\|M\,v\|\,; \ v\in \C^n,\ \|v\|=1\}$ is the operator norm of a matrix $M\in \C^{m,n}$. The set $W(A):=\{v^*Av\,; v\in \C^d, \|v\|^2=v^*v=1\}$ is the numerical range 
of $A\in \C^{d,d}$\,; recall that it is a closed convex subset of $\C$ which contains the spectrum $\sigma (A)$ (see for instance \cite{gura,hojo}). The inequality  
 \[
\|A\|\leq 2 \sup_{z\in W(A)}|z|,
\]
is well known. In \cite{crzx}, we have conjectured that its extension 
  \begin{equation}\label{conje}
\|p(A)\|\leq 2 \sup_{z\in W(A)}|p(z)|,
\end{equation}
 also holds for all polynomials $p\in \C[z]$. But, up to now, we have only been able to prove in \cite{crac} that there exists a best constant $\q$, satisfying $2\leq \q\leq 11.08$, such that the inequality
\begin{equation}\label{cst}
\|p(A)\|\leq \q \sup_{z\in W(A)} |p(z)|
\end{equation}
holds for all square matrices $A$ and all polynomials $p$. 
It is remarkable that the inequality (\ref{cst}) admits a completely bounded version\footnote{
Some comments are provided at the end of this section}. More precisely\,:

{\it There exists a best constant  $\qcb$, satisfying $2\leq \q\leq \qcb\leq 11.08$, such that}
\begin{equation}\label{ccb}
\|P(A)\|\leq \qcb \sup_{z\in W(A)} \|P(z)\|,
\end{equation}
\indent $\bullet$ {\it for all square matrices $A\in\C^{d,d}$, for all values of $d$,}

$\bullet$ {\it for all polynomial functions $P\,: \C\to\C^{m,n}$, for all values of $m$ and $n$.}
\smallskip

\noindent Here $P$ is matrix-valued $P(z)=(p_{ij}(z))$, with each entry $p_{ij}\in\C[z]$ being a polynomial\,; the matrix $P(A)\in\C^{md,nd}$ is constituted of $m\times n$ blocks of size $d\times d$, the $(i,j)$-th block being
$p_{ij}(A)$.\medskip

The surprising fact is the existence of such uniform bounds  $\q$ and $\qcb$, independently of the matrix $A$, of its size, of the degree of polynomials used,
as well as of $m$ and $n$ for $\qcb$. 
This universality allows us to extend the inequalities to any bounded linear operator $A\in{\cal L}(H)$ on a complex Hilbert space
$H$ (and even to unbounded operators), and also to any continuous function $p$ (resp. $P$) on $\overline{W(A)}$ which is holomorphic in the interior of the numerical range. We refer to \cite{crac,crzx1} for these extensions and for some applications.\medskip

Note that, in the case of a normal matrix $A$, we have better estimates
$\| p(A)\|\leq\sup_{z\in \sigma(A)}|p(z)|$ and $\|P(A)\|\leq \sup_{z\in \sigma(A)} \|P(z)\|$,
where $\sigma(A)$ denotes the spectrum of $A$ (it is well known that $\sigma(A)\subset W(A)$.)
Inequalities \eqref{cst} and \eqref{ccb} are of interest since they provide estimates for non-normal matrices.
More generally, if $A=X^{-1}NX$ is similar to a normal matrix $N$, we easily get 
$\| p(A)\|\leq\|X\|\,\|X^{-1}\|\sup_{z\in \sigma (A)}|p(z)|$ and $\|P(A)\|\leq \|X\|\,\|X^{-1}\|\sup_{z\in \sigma (A)}\|P(z)\|$. Thus, in this case, the inequality \eqref{conje} (as well as its completely bounded version $\|P(A)\|\leq 2\sup_{z\in W(A)} \|P(z)\|$) holds if the condition number of $X$ satisfies $\|X\|\,\|X^{-1}\|\leq 2$.
\medskip

There exist some other cases where we know that the inequality \eqref{conje} (as well as its completely bounded version) is satisfied:

$\bullet$ {\it If $p(z)=(z{-}z_0)^n$.} It suffices to consider the case $z_0=0$. Then, a result of Okubo and Ando \cite{okan} implies that $A=X^{-1}BX$ with $\|B\|\leq w(A):=\max\{|z|\,; z\in W(A)\}$ and $\|X\|\,\|X^{-1}\|\leq 2$. The inequalitiy \eqref{conje} and its completely bounded version then follow from a von Neumann inequality\,\cite{vn, paul}. \smallskip

$\bullet$ {\it If $W(A)$ is a disk}. The proof \cite{bacrde} uses the same argument of Okubo and Ando.\smallskip

$\bullet$ {\it In dimension $d=2$.} See \cite{crzx} and \cite{bacrde}. The proof uses a similarity transformation and the knowledge of the conformal map from the numerical range (here an ellipse) onto the unit disk. \smallskip

$\bullet$ {\it If $A$ is a quadratic matrix}. This means that $(A{-}z_1I)(A{-}z_2I)=0$
for some complex numbers $z_1$ and $z_2$. Then, $A$ is unitarily similar to a direct sum of $1\times1$ and $2\times 2$ matrices with eigenvalues $z_1$ or $z_2$ or $z_1$ and $z_2$, the numerical range is an ellipse (cf.  \cite{tw} Theorem 1.1), and the inequality follows from the case $d=2$. \smallskip

$\bullet$ {\it If $d=3$ and $A^3=0$.} See \cite{crzx3}; the argument is not fully mathematical, but uses also a small computational part. \smallskip

$\bullet$ {\it If $A=PD$ where $P$ is a permutation matrix and $D$ is diagonal.} The proof has been obtained by Daeshik Choi \cite{choi}.
In this situation, the numerical range has the symmetries of the regular $d$-sided polygon
and the image of $A$ by a conformal mapping from $W(A)$ onto the unit disk has the form $c\,A$.\smallskip

$\bullet$ {\it In infinite dimension\,: if $W(A)$ contains a sector of angle $2\alpha\geq \frac{2\pi }{3}$.} Then, it is known \cite{bacrde} that  $\|R(A)\|\leq \frac{\pi -\alpha}{\alpha}\sup_{z\in W(A)} \|R(z)\|$ for all rational functions bounded in $W(A)$. Thus \eqref{conje} and \eqref{ccb} hold for
$\alpha\geq \frac\pi 3$.

 \bigskip

\noindent{\it Some open problems.} My conjecture $\q=2$ and its strong form $\qcb=2$, or obtaining the exact values of $\q$ and $\qcb$, seem to me exciting, but difficult, open problems. At least, a challenging question will be to sharply improve the upper bound $11.08$. Our proof of this estimate is quite involved and clearly not optimal. It is not clear to me whether $\q=\qcb$.
In case of a positive answer, it will be interesting to understand the difference between this situation and the general context of polynomial bounds and complete bounds. (It is  known  \cite{cgh} that they can be different if $d\geq 3$.)
\medskip

In order to consider easier problems, we may try to bound constants related to subfamilies of matrices
\begin{align*}
\q(d)&:=\sup_{A,p}\{\|p(A)\|\,; \ A\in\C^{d,d},\ p\in \C[z], \ |p(z)|\leq1\text{ in }W(A)\},\hskip1cm\\
\qcb(d)&:=\sup_{A,P,m,n}\{\|P(A)\|\,;\  A\in\C^{d,d},\ P\in \C^{m,n}[z] , \ \|P(z)\|\leq1\text{ in }W(A)\},\\
\psi(A)&:=\sup_{p}\{\|p(A)\|\,; \ p\in \C[z], \ |p(z)|\leq1\text{ in }W(A) \},\\\
\pc(A)&:=\sup_{P,m,n}\{\|P(A)\|\,;\  P\in\C^{m,n}[z] , \ \|P(z)\|\leq1\text{ in }W(A)\}.
\end{align*}
It is easily verified that $\q$ and $\qcb$ are non-decreasing with $d$\,; furthermore $\q=\sup_{d}\q(d)$ and $\qcb=\sup_{d}\qcb(d)$. Clearly $\q(1)=\qcb(1)=1$; we have succeeded to show \cite{bacrde}  that $\q(2)=\qcb(2)=2$, but failed with the questions $\q(3)=\qcb(3)$ and $\q(3)=2$\,; a fortiori the analogue questions are open for $d>3$.
(The numerical experiments seem to confirm that $\q(3)=2$, but we have only succeeded to prove that $\q(3)\leq 9.995$.)

In Section\,2, we will see that the bounds
\[
\q(d)=\max\{\psi(A)\,; A\in\C^{d,d}\},\quad
\q_{cb}(d)=\max\{\pc(A)\,; A\in\C^{d,d}\},
\]
are realized and that, if all eigenvalues of $A$ are in the interior of $W(A)$, then $\psi $ and $\pc$ depend continuously on $A$.

\medskip

In Section\,3, we consider constants related to  the family of matrices with numerical range contained in a non-empty convex domain $\varOmega\neq\C$ of  the complex plane, not necessarily bounded. We set
\[
C(\varOmega,d):=\sup_{A,r}\{\Vert r(A)\Vert ; A\in\C^{d,d}, \ W(A)
\subset\varOmega ,\ r:\C\to\C,\ |r(z)|\leq 1,\forall z\in
\varOmega \},
\]
\begin{align*}
C_{cb}(\varOmega,d):=\sup_{A,R,m,n}\{\Vert R(A)\Vert ;  A\in\C^{d,d}, \ W(A)
\subset\varOmega,\hskip3cm\\ R:\C\to\C^{m,n},\ \Vert R(z)\Vert \leq 1,\forall z\in
\varOmega \},
\end{align*}
\[
C(\varOmega):=\sup_{d}C(\varOmega,d),\qquad C_{cb}(\varOmega):=\sup_{d}C_{cb}(\varOmega,d).
\]
In these definitions, $r$ and $R$ denote rational functions. (This choice has been made for treating together the bounded and unbounded domain cases, but for a bounded $\varOmega$ it would have sufficed to only consider polynomials $r$ and $R$ without change of the values. Similarly, the condition $ W(A)\subset\varOmega$ could be replaced by $ W(A)\subset\overline\varOmega$.)
\medskip
Clearly, there holds
\[
\q(d)=\sup_{\varOmega} C(\varOmega,d), \quad \q=\sup_{\varOmega} C(\varOmega), \quad
\q_{cb}(d)=\sup_{\varOmega} C_{cb}(\varOmega,d), \quad \q_{cb}=\sup_{\varOmega} C_{cb}(\varOmega).
\]
We review some results concerning these constants. In Section\,4,
we give some lower bounds for $C(\varOmega,d)$,  while Section\,5 is concerned with their realization.
In Section\,6, I give some personal comments on the interest in the numerical range and in Section\,7, I provide some arguments supporting my conjecture. Sections\,8 and 9 
describe some of our numerical experiments realized with the open source software SCILAB.
Section\,10 is devoted to matrices with the unit disk as numerical range and realizing $\psi (A)=2$.
Finally, in Section\,11, we conclude by a list of open problems and some final comments.
\smallskip

\noindent{\it About the complete bound.}
Let us consider the map $u_A$\,: $p\mapsto p(A)$ from the algebra of polynomials $\C[z]$ (equipped  with the maximum norm on $W(A)$) into the algebra of $d\times d$ matrices. Clearly, inequality \eqref{cst} means that the map $u_A$ is bounded with constant $\q$. Inequality \eqref{ccb} is the tensorial version of \eqref{cst}; by definition, $\qcb$ is called the complete bound of $u_A$. The notion of completely bounded maps is defined in a more general context and plays an important role in operator theory; these maps are the natural morphisms in the category of operator spaces and have been the subject of extensive studies since the early 80's; see for instance the books\,\cite{efru,paul}. I am convinced of their interest in an applied situation. For instance, in their pioneering work\,\cite{dede}, Bernard and Fran\c{c}ois Delyon have shown the usefulness of the numerical range by solving the Burkholder conjecture. For that, they have established the estimate
\[
\sum_{n\geq 1}n\|T^n{-}T^{n-1}\|^2\leq C(W(T))\sup_{z\in W(T)}\sum_{n\geq 1}n|z^n{-}z^{n-1}|^2.
\]
This estimate now corresponds to \eqref{ccb} with $n=1$, $m\to\infty$, 
\[
P(z)=(z{-}1,\dots,\sqrt n (z^n{-}z^{n-1}),\dots)^T\quad\text{and}\quad
C(W(T))\leq \qcb.
\]
Another context where formulations using matrix-valued polynomial (resp.\,rational) functions of a matrix naturally occur is the discretization of linear differential systems by explicit (resp.\,implicit) linear multistep methods.
For instance, if we discretize the pendulum equations $\dot{p}=A\,q$, $\dot{q}=p$ by the St\"ormer-Verlet scheme with a stepsize $h$, we get
\[
p_{n+1/2}=p_{n-1/2}+hAq_n,\quad q_{n+1}=q_n+hp_{n+1/2}.
\]
Therefore
\[
\begin{pmatrix}q_n\\p_{n+1/2}\end{pmatrix}=
\begin{pmatrix}1 &h\\hA &1+h^2A\end{pmatrix}^n
\begin{pmatrix}q_0\\p_{1/2}\end{pmatrix}
.
\]
\section{Bounds related to a matrix} 
We first note that $\psi(\alpha A{+\beta I})=\psi (A)$, $\pc(\alpha A{+\beta I})=\pc(A)$ if  $\alpha,\beta \in\C$ and $\alpha\neq 0$, and
$\psi(U^*AU)=\psi (A)$, $\pc(U^*AU)=\pc(A)$, if $U$ is a unitary matrix.
\begin{theorem}
 The maps $A\mapsto \psi (A)$ and $A\mapsto \pc (A)$ are continuous in the set of matrices with
 all eigenvalues of $A$  in the interior of the numerical range.
 \end{theorem}
\begin{proof} We only consider the case $\psi $, the other case being similar.
Assume that $A_n\to A$ as $n\to\infty$ and that all eigenvalues of $A$ are in the interior of the numerical range. Without loss of generality, we can assume that $0$ is interior to $W(A)$, and that $0$ and 
all eigenvalues of $A_n$ are interior to $W(A_n)$. Then, there exists a sequence $\lambda_n\to 1$ such that 
\[
\tfrac{1}{\lambda _n}W(A_n)\subset W(A)\quad\text{and}\quad \tfrac{1}{\lambda _n}W(A)\subset W(A_n).
\]
Let us consider $p\in \C[z]$ with $|p|\leq 1$ in $W(A)$; then
\begin{align*}
\|p(A)\|&\leq \|p(A)-p(\tfrac{1}{\lambda _n}A_n)\|+\|p(\tfrac{1}{\lambda _n}A_n)\|\\
&\leq \|p(A)-p(\tfrac{1}{\lambda _n}A_n)\|+\psi (A_n),
\end{align*}
since $|p|\leq 1$ in $W(\tfrac{1}{\lambda _n}A_n)\subset W(A)$ and $\psi (\tfrac{1}{\lambda _n}A_n)=\psi (A_n)$. Furthermore,
\[
p(A)-p(\tfrac{1}{\lambda _n}A_n)=\frac{1}{2\pi i}\int_{\partial W(A)}p(\sigma )\big((\sigma{ -}A)^{-1}-(\sigma {-}\tfrac{1}{\lambda _n}A_n)^{-1})\big)\,d\sigma,
\]
whence
\[
\|p(A)-p(\tfrac{1}{\lambda _n}A_n)\|\leq \ \varepsilon _n:= \frac{1}{2\pi}\int_{\partial W(A)}\big\|(\sigma {-}A)^{-1}-(\sigma {-}\tfrac{1}{\lambda _n}A_n)^{-1}\big\|\,|d\sigma|.
\]
Therefore, there holds $\psi (A)\leq \varepsilon _n +\psi (A_n)$; similarly we obtain
$\psi (A_n)\leq \tilde\varepsilon _n +\psi (A)$ with
\[
\tilde\varepsilon _n:= \frac{1}{2\pi}\int_{\partial W(A_n)}\big\|(\sigma {-}A_n)^{-1}-(\sigma {-}\tfrac{1}{\lambda _n}A)^{-1}\big\|\,|d\sigma|.
\]
It is easily seen that $\varepsilon_n\to0$ and $\tilde\varepsilon_n\to0$ as $n\to \infty$, which shows the desired continuity.
\end{proof}\medskip

{\sc Remark.} \emph{It is also possible to show the continuity for matrices with distinct eigenvalues, but the continuity of $\psi $ is not generally true, for instance}
$\psi\begin{pmatrix}0 &0\\0 &0\end{pmatrix}=1$ and $\psi\begin{pmatrix}0 &\varepsilon\\0 &0 \end{pmatrix}=2$. {\it Also, using $p(z)=\frac{1-z^{\pi /2\varepsilon }}{1+z^{\pi /2\varepsilon }}$, it can be seen that}
\[
\psi \begin{pmatrix}0&0 &0\\0&1 &0\\0 &0&1\end{pmatrix}=1\quad\text{and}\quad
\psi\begin{pmatrix}0&0 &0\\0 &1&2\sin\varepsilon\\0 &0&1 \end{pmatrix}\geq \frac{\pi \sin\varepsilon }{2\varepsilon }.
\]\medskip

\begin{theorem}\label{att}
 There exists a matrix $A\in\C^{d,d}$ such that $\q(d)=\psi (A)$ and a matrix $A\in\C^{d,d}$ such that $\qcb(d)=\pc(A)$.
\end{theorem}
\begin{proof} We only consider the polynomial bound  $\q(d)$, the proof for $\qcb$ being similar.
 This is clearly the case if $d=1$, or if $\q(d)=2$. We assume $d\geq 3$ and $\q(d)>\q(d{-}1)$
 (the case  $\q(d)=\q(d{-}1)$ may be treated by recursion). Then there exists a sequence of matrices $A_n\in\C^{d,d}$ with $\psi (A_n)\to \q(d)$. Without loss of generality, we may assume that $\|A_n\|=1$, trace$(A_n)=0$, and $A_n\to A$. If all eigenvalues of $A$ are in the interior of $W(A)$ then, according to the previous theorem, $\psi (A)=\q(d)$. It suffices to show that the opposite case $\sigma (A)\cap\partial W(A)\neq \emptyset$ is impossible. Indeed, otherwise $A$ will be unitarily similar to a block diagonal matrix
 $\begin{pmatrix}\lambda I &0\\0 &C\end{pmatrix}$ with $\lambda $ eigenvalue with multiplicity $k$, $\lambda \notin \sigma(C)$. Note that $\|A\|=1$ and trace$(A)=0$ induces $k<d$. Replacing $A$ by $U^*AU$ and $A_n$ by $U^*A_nU$ if needed, we may assume from now on that $A=\begin{pmatrix}\lambda I &0\\0 &C\end{pmatrix}$. For $n$ large enough, $A_n$ has exactly $k$ eigenvalues inside a disk $\{z\,;|z{-}\lambda|<r\}$, the others being strictly outside. We introduce the projector $P_n=\frac{1}{2\pi i}\int_{|z{-}\lambda|=r} (zI{-}A_n)^{-1}dz$
 on the corresponding invariant subspace and the matrix $M_n$ with column vectors $m_j=P_ne_j$, if $j\leq k$, and  $m_j=(I{-}P_n)e_j$, if $j>k$ ($\{e_1,e_2,\dots,e_d\}$ denoting the canonical basis of $\C^d$).
 Then, $M_n^{-1}A_nM_n=\begin{pmatrix}B_n &0\\0 &C_n\end{pmatrix}$ is block diagonal and $\lim_{n\to\infty}M_n=I$. Let us use the factorization $M_n=Q_nR_n$ with $Q_n$ unitary and $R_n$ upper triangular with positive diagonal; then  $\lim_{n\to\infty}Q_n=I$ and $\lim_{n\to\infty}R_n=I$.
 The matrix $\tilde A_n=Q_n^*A_nQ_n$ satisfies
\[
\tilde A_n=Q_n^*A_nQ_n=R_n \begin{pmatrix}B_n&0\\0&C_n\end{pmatrix}R_n^{-1}=
\begin{pmatrix}\tilde B_n&\tilde D_n\\0&\tilde C_n\end{pmatrix},
\]
 with $\lim_{n\to\infty}\tilde B_n=\lambda I$, $\lim_{n\to\infty}\tilde C_n=C$ and $\lim_{n\to\infty}\tilde D_n=0$. Since the spectra of $\tilde B_n$ and $\tilde C_n$ are disjoint for $n$ sufficiently large, the Sylvester equation
 \[
-\tilde B_n E_n+E_n\tilde C_n=\tilde D_n
\]
 has a unique solution $E_n\in \C^{k,n-k}$ and $\lim_{n\to\infty}E_n=0$. Now, we note that
 \[
\tilde A_n= \begin{pmatrix}I&E_n\\0&I\end{pmatrix}
\begin{pmatrix}\tilde B_n&0\\0&\tilde C_n\end{pmatrix}
\begin{pmatrix}I&-E_n\\0&I\end{pmatrix};
\]
 therefore, if $p\in\C[z]$ satisfies $|p|\leq 1$ in $W(\tilde A_n)$,
 \[
p(\tilde A_n)= \begin{pmatrix}I&E_n\\0&I\end{pmatrix}
\begin{pmatrix}p(\tilde B_n)&0\\0&p(\tilde C_n)\end{pmatrix}
\begin{pmatrix}I&-E_n\\0&I\end{pmatrix},
\]
and $\|p(\tilde A_n)\|\leq (1{+}\|E_n\|)^2\max(\|p(\tilde B_n)\|,\|p(\tilde C_n)\|)$. But, $|p|$ is also bounded by 1 in 
$W(\tilde B_n)$ and in $W(\tilde C_n)$ (subsets of $W(\tilde A_n)$); thus 
$\|p(\tilde B_n)\|\leq \q(k)$ and $\|p(\tilde C_n)\|\leq \q(d{-}k)$. This yields $\psi (\tilde A_n)\leq (1{+}\|E_n\|)^2\max(\q(k),\q(d{-}k))$. Then, noticing that $\psi (A_n)=\psi (\tilde A_n)$ and passing to the limit, we get $\q(d)\leq\max(\q(k),\q(d{-}k))$, which contradicts $\q(d)>\q(d{-}1)$. 
 \end{proof}\medskip
 
We turn now to the realization of $\psi (A)$. If the matrix $A$ is normal, then $\psi (A)=1$ and this bound is realized by $p(z)=1$.
We may thus assume that the matrix $A$ is not normal; then the interior of $W(A)$ is non empty, and there exists a conformal map $a$ from the interior of $W(A)$ onto the open unit disk $\D$. By density of the space $\C[z]$ (Mergelyan's theorem) it is easily seen that
\[
\psi(A)=\sup_{f}\{\|f(A)\|\,; \ f\in C(W(A))\cap {\mathcal H}(\stackrel{\circ}{W(A)}), \ |f(z)|\leq1\text{ in }W(A) \},
\]
($\mathcal H(\varOmega)$ denotes the set of holomorphic function in the open $\varOmega$).
We set now $B=a(A)$ and $f=g\circ a$. Then $f(A)=g(B)$, and clearly $\psi (A)=\psi _{\D }(B)$ with
\[
\psi_{\D}(B)=\sup_{g}\{\|g(B)\|\,; \ g\in C(\overline\D)\cap {\mathcal H}(\D), \ |g(z)|\leq1\text{ in }\D) \}.
\]
If $A$ has no eigenvalue on the boundary of $W(A)$, then 
the eigenvalues of the matrix $B$ are in the open unit disk and $\psi_{_\D}(B)$ is attained by a Blaschke product with at most $d{-}1$ terms  \cite{crzx}. More precisely, we have 
\begin{equation}\label{psid}
\psi_{_\D}(B):=\max_{\zeta_{j}}\{\|g(B)\|\,; g(z)=\prod_{j=1}^{r}
\frac{z-\zeta_j}{1-\bar\zeta_j z}, \ \zeta_{1},\dots,\zeta_{r}\in \D,\ r\leq d-1\},
\end{equation}
and then by setting $b_{z_j}(z)=(a(z){-}a(z_j)/(1{-}\overline{a(z_j)}a(z))$
\[
\psi(A)=\max\{\Big\|\prod_{j=1}^{r} b_{z_j}(A)\Big\|\,; \ z_1,\dots,z_r\in\  \stackrel{\circ}{W(A)},\ r\leq d-1\}.
\]
If $A$ has $k$ eigenvalues on the boundary of $W(A)$, we get easily the same result, with $r\leq d{-}1$ replaced by $r\leq d{-}k{-}1$. 
\medskip

For the completely bounded analogue quantity, a Paulsen theorem \cite{paul}
provides the characterization
\begin{equation}\label{psicbd}
\psi_{_{cb,\D}}(B):=\min_{S}\{\|S\|\,\|S^{-1}\|\,; S\in\C^{d,d}, \|S^{-1}BS\|\leq1\}.
\end{equation}
We deduce
\[
\pc(A):=\min_{S}\{\|S\|\,\|S^{-1}\|\,; S\in\C^{d,d}, \|S^{-1}a(A)S\|\leq1\}.
\]

\section{Constants related to a convex domain}

In all this paper, by convex domain $\varOmega$, we mean a convex open subset of the complex plane, not necessarily bounded, such that $\varOmega\neq \emptyset$ and $\varOmega\neq\C$. Recall that, with $r$ and $R$ denoting rational functions,
\[
C(\varOmega,d):=\sup_{A,r}\{\Vert r(A)\Vert ; A\in\C^{d,d}, \ W(A)
\subset\varOmega ,\ r:\C\to\C,\ |r(z)|\leq 1,\forall z\in
\varOmega \},
\]
\begin{align*}
C_{cb}(\varOmega,d):=\sup_{A,R,m,n}\{\Vert R(A)\Vert ;  A\in\C^{d,d}, \ W(A)
\subset\varOmega ,\hskip3cm\\\ R:\C\to\C^{m,n},\ \Vert R(z)\Vert \leq 1,\forall z\in
\varOmega \},
\end{align*}
\[
C(\varOmega):=\sup_{d}C(\varOmega,d),\qquad C_{cb}(\varOmega):=\sup_{d}C_{cb}(\varOmega,d).
\]

{\sc Remarks.}

 1) {\it The previous constants depend  only on $d$ and on the shape of $\varOmega
$. More precisely, if $\varphi $ is linear\,: $\varphi (z)=a\!+\! 
b\,z,$
or conjugate linear\,: $\varphi (z)=a\!+\! b\,\bar z$, $\ a,b\in \C,\ 
b\neq
0$, then we have $C(\varOmega )=C(\varphi (\varOmega )),\dots,C_{cb}(\varOmega,d )= 
C_{cb}(\varphi
(\varOmega ),d)$.}
\medskip

2) {\it A classical result of J. von Neumann \cite{vn} asserts that $C(\varOmega)=1$, if $\varOmega$ is a half-plane\,; as soon as the notion of ``completely bounded'' has appeared, it has been remarked that in this case we also have $C_{cb}(\varOmega)=1$. }

3) {\it Obviously $C(\varOmega )\leq C_{cb}(\varOmega )$ and $C(\varOmega,d)\leq
C_{cb}(\varOmega,d)$. Furthermore, the last two constants are increasing (perhaps not strictly) functions of $d$.
}\medskip

Except for the old half-plane inequality, the first result on this subject is quite recent. In the nice paper \cite{dede} an estimate $C(\varOmega)<+\infty$ is given for any bounded convex domain $\varOmega$\,; in \cite{bacrde,crac} we have shown that this result is still valid in completely bounded form, and improved the estimate to
\[
C_{cb}(\varOmega )\leq \min\big(11.08,2+\pi +\inf_{\omega  \in\, \partial\varOmega }{\rm
TV}(\log|\sigma\! -\!\omega |)\big)\,;
\]
here TV$(\log|\sigma\! -\!\omega |)$ is the total variation of
$\log(|\sigma\! -\!\omega |)$ as $\sigma $ traces $\partial \varOmega$. (A slightly better estimate can be deduced from Lemma 9 in \cite{crac}.)

\smallskip A similar approach provides the inequality
\[
C_{cb}(\varOmega )\leq \min\big(11.08,\ 1+\frac{2}{\pi }\int_\alpha^{\pi /2}\frac{\pi 
-x+\sin x}{\sin x}\, dx,\ \frac{\pi -\alpha}{\alpha}\big),
\]
if $\varOmega $ contains a sector with angle $2\alpha $,
$0<\alpha \leq \tfrac{\pi }{2}$. For a sector $\varOmega
=S_\alpha $ (with angle $2\alpha \leq \pi $) we have obtained the more precise estimates
\cite{crde}, \cite{bacrde}, \cite{bc},
\[
\frac{\pi\,\sin\alpha}{2\alpha}\leq C(S_{\alpha})\leq C_{cb}(S_\alpha  )\leq \frac{\pi -\alpha }{\pi }\Big( 2{-}\frac{2}{\pi}\log\tan\big(\frac{\alpha \,\pi}
{4(\pi \!-\! \alpha)}\big)\Big),\quad{ \rm for } \  \alpha\in (0,\pi/2],
\]
and
\[
C(S_{\alpha})\leq C_{cb}(S_\alpha  )\leq 2{-}\frac{2\alpha}{\pi}{+}\frac{2\,\cos \alpha}{\pi\,\sqrt{1{+}2\cos 2\alpha}}\arccos \Big(\frac{\cos(\pi{-}2\alpha)}{\cos \alpha}\Big), \quad
{ \rm for } \  \alpha\in [0,\pi/3].
\]
The second bound is better than the first if $\alpha\leq.22\,\pi$ and is still valid if we replace the sector
$S_{\alpha}$ by (a domain limited by) a branch of hyperbola of angle $2\alpha$.
In \cite{bc} we derived the bound $C_{cb}({\cal E})\leq 2+2/\sqrt{4\!-\!e^2}$
for an ellipse ${\cal E}$ of eccentricity $e$ and $C_{cb}({\cal P})\leq 2\!+\!2/\sqrt3$ for a parabola
${\cal P}$. The estimate $C_{cb}(S_{0})\leq 2\!+\!2/\sqrt3$ is also known for a strip $S_{0}$, \cite{crde}.
  
 The exact values known are for the half-plane case $C(\varPi )=C_{cb}(\varPi )=1$  and for the disk case $C(D )=C_{cb}(D )=2$, see  \cite{bacrde}\,; the values of $C(S_{\alpha},2)$ are also known, see  \cite[Theorem\,4.2]{crzx}. The other bounds, and in particular the general bound
$C_{cb}(\varOmega) \leq 11.08$, are very pessimistic. 

In order to characterize these constants it may be useful to introduce a conformal mapping $a$ from
$\varOmega$ to $\D$ and to set
 $b_{z_j}(z)=(a(z){-}a(z_j)/(1{-}\overline{a(z_j)}a(z))$. Then
\[
C(\varOmega,d)=\max_{A,z_1,\dots,z_r}\{\Big\|\prod_{j=1}^{r} b_{z_j}(A)\Big\|\,; A\in\C^{d,d},\ W(A)\subset \varOmega,\ z_1,\dots,z_r\in\varOmega,\ r\leq d-1\},
\]
\[
C_{cb}(\varOmega,d):=\max_{A\in\C^{d,d}}\{\min_{S\in\C^{d,d}}\{\|S\|\,\|S^{-1}\|\,;  \|S^{-1}a(A)S\|\leq1\}\,; \ W(A)\subset \varOmega\}.
\]

We turn now to the lower semi-continuity.
\begin{lemma}\label{sci}
 If we have a sequence of convex domains such that $\varOmega\subset\varOmega_n\subset(1{+}\varepsilon _n)\varOmega$ with $\varepsilon _n\to 0$ as $n\to \infty$, then $\liminf C(\varOmega_n)\geq C(\varOmega)$, $\liminf C_{cb}(\varOmega_n)\geq C_{cb}(\varOmega)$, $\liminf C(\varOmega_n,d)\geq C(\varOmega,d)$, and $\liminf C_{cb}(\varOmega_n,d)\geq C_{cb}(\varOmega,d)$.
\end{lemma}
\begin{proof}
 We just look at the first inequality and assume $\varOmega$ bounded. Let us consider $c<C(\varOmega)$. There exists a polynomial $p$ with $\sup_{\varOmega}|p|=1$, and a matrix $A$ with $W(A)\subset \varOmega$ such that $\|p(A)\|\geq c$. Then, $W(A)\subset \varOmega_n$, whence
 \begin{align*}
c\leq \|p(A)\|\leq C(\varOmega_n) \sup_{z\in\varOmega_n}|p(z)|\leq C(\varOmega_n) \sup_{z\in\varOmega}|p((1{+\varepsilon _n)}z)|.
\end{align*}
We deduce $\liminf C(\varOmega_n)\geq c$ from $\lim\sup_{z\in\varOmega}|p((1{+\varepsilon _n)}z)|=
\sup_{z\in\varOmega}|p(z)|$. The same proof leads to the other inequalities, in the unbounded case
as well, by replacing polynomials by rational functions.
\end{proof}

\section{Some lower bounds on the constants}
We begin with a general lower bound

\begin{lemma}
Let $\varOmega$ be a convex domain such that $D\subset\varOmega\subset S$, where $D$ is an open disk and $S$ an open cone with tangent sides to $D$. Then, there holds 
$C(\varOmega,2)\geq \frac{\pi \sin\alpha}{2\,\alpha}$, where $2\,\alpha$ is the aperture angle of $S$.
Similarly, if $D\subset\varOmega\subset S$ with $S$ an open strip with tangent sides to $D$, the following lower bound $C(\varOmega,2)\geq \frac{\pi}{2}$ holds.
\end{lemma}
\begin{proof}{\it Cone case}.
We can assume that $S_{\alpha}=\{z\in\C\,; z\neq0$, $|\arg z|<\alpha\}$ and $D=\{z=1+\rho \sin\theta \,; 0\leq \rho <\sin\alpha, \theta \in \R\}$. Then, we choose
\[
A=\begin{pmatrix}1 &2\sin\alpha\\0 &1\end{pmatrix}, \quad 
f(z)=\frac{1 -z^{\pi/2\alpha}}{1+z^{\pi/2\alpha}},\quad\text{so that\quad}
f(A)=2\sin\alpha \,f'(1)
\begin{pmatrix}0 &1\\0 &0\end{pmatrix}.
\]
Notice that $W(A)\subset\overline{D}\subset\overline{\varOmega}$ and $|f(z)|\leq 1$ in $S$; thus, a fortiori, $|f(z)|\leq 1$ in $\varOmega$. This yields
\[
\frac{\pi \sin\alpha}{2\,\alpha}=\|f(A)\|\leq C(\varOmega,2).
\]

\noindent{\it Strip case}.
We can assume $S_0=\{z\in\C\,; |\rm{Im}\,z|<1\}$ and $D=\{z=\rho \sin\theta \,; 0\leq \rho <1, \theta \in \R\}$. Then we choose
\[
A=\begin{pmatrix}0 &2\\0 &0\end{pmatrix}, \quad 
f(z)=\tanh(\tfrac\pi 4 z),\quad\text{so that\quad}
f(A)=\tfrac\pi 2\,\begin{pmatrix}0 &1\\0 &0\end{pmatrix}.
\]
We conclude as previously.
\end{proof}\medskip

{\sc Remark.} {\it This may be applied with $\varOmega=S_\alpha$. Thus
$C(S_\alpha,2)\geq \frac{\pi \sin\alpha}{2\,\alpha}$ and $C(S_0,2)\geq \frac\pi 2$. The exact value of $C(S_\alpha,2)$ is known \cite{crzx} and it is quite close, for example for $\alpha=0$, $C(S_0,2)=1.58765\dots$}\medskip

\begin{corollary}
 If $\varOmega$ is a bounded convex domain, then $C(\varOmega,2)\geq 1.5$.
 If $\varOmega$ is an unbounded convex domain with aperture at infinite $2\alpha$, then $C(\varOmega,2)\geq\frac{\pi \sin\alpha}{2\alpha}$, ( $\geq \frac\pi 2$ if $\alpha=0$).
 If $\varOmega\neq \C$ is a convex domain and if $C(\varOmega,2)=1$, then $\varOmega$ is a half-plane.
\end{corollary}
\begin{proof}
 {\it Bounded case.} There exists a largest disk $D$ contained in $\varOmega$. There are two possible situations: 
 
 - either the boundaries of $\varOmega$ and $D$ have two contact points which are diametrically opposite in $D$. Then, $\varOmega$ is located between a disk $D$ and a strip, whence $C(\varOmega,2)\geq \frac\pi 2$. 
 
 - or $\varOmega$ is located between a disk $D$ and (at least) a triangle tangent to $D$.
 Let $2\alpha$ be the smallest angle of this triangle; then $C(\varOmega,2)\geq\frac{\pi \sin\alpha}{2\alpha}\geq \frac32$, since $\alpha\leq \frac\pi 6$ and $\frac{\pi \sin\alpha}{2\alpha}$ is a decreasing function of $\alpha$.\smallskip
 
 \noindent{\it Unbounded case}. Then, there exist a sequence of disks $D_n$ and cones $S_n$ 
 with aperture $2\alpha_n$ and sides tangent to $D_n$ such that $D_n\subset\varOmega\subset S_n$ and $\alpha_n\to\alpha$ as $n\to\infty$. Therefore, $C(\varOmega,2)\geq\frac{\pi \sin\alpha}{2\alpha}$. The value 1 only occurs if $\alpha=\frac\pi 2$; then $\varOmega$ is a half-plane.
\end{proof}

{\sc Remark.} {\it In the unbounded case, the two-sided estimate $\frac{\pi \sin\alpha}{2\alpha}\leq C(\varOmega,d)\leq \frac{\pi -\alpha}{\alpha}$ holds; this is quite sound for $\alpha$ close to $\pi /2$.}\medskip

\begin{lemma}
Let $\varOmega$ be a convex domain with an angular point on the boundary of aperture $2\alpha\in]0,\pi [$. Then, $C(\varOmega,d)\geq C(S_{\alpha},d)$.
The same lower bound holds if $\varOmega$ is a convex domain which is asymptotic at infinity to a cone of aperture
$2\alpha$.
\end{lemma}
\begin{proof}
{\it Bounded case}.
Without loss of generality, we may assume that the angular point is the origin $0$ and that $\varOmega$ is contained in the sector $S_{\alpha}=\{z\in\C\,; z\neq0$ and $|\arg z|<\alpha\}$. Let us choose $\gamma<C(S_{\alpha},d)$\,; there exists a matrix $A\in \C^{d,d}$, with $W(A)\subset S_{\alpha}$, and a rational function $r$ bounded by 1 in $S_{\alpha}$, such that $\|r(A)\|\geq \gamma$. Note that, for all $\varepsilon>0$, $W(\varepsilon A)\subset S_{\alpha}$ and $r_{\varepsilon}(z)=r(z/\varepsilon)$ is still bounded by 1 in $S_{\alpha}$. But, for $\varepsilon$ small enough, 
$W(\varepsilon A)\subset \varOmega$ and, since $r_{\varepsilon}$ is a fortiori bounded by 1 in $\varOmega$, this shows $\gamma\leq\|r(A)\|=\|r_{\varepsilon}(\varepsilon A)\|\leq C(\varOmega,d)$.

{\it Unbounded case}. We can assume that $\varOmega$ is asymptotic to $S_\alpha$ with $\varOmega\subset S_\alpha$. The proof is similar by letting $\varepsilon \to \infty$. 
\end{proof}\medskip

\begin{lemma}
 Let $\varOmega\neq \C$ be a convex domain of $\C$ and $a$ a conformal mapping from $\varOmega$ onto the unit disk $\D$. Then, there holds 
\[C(\varOmega,2)\geq 2\sup_{{z_{1}\in\varOmega}} \frac{d(z_{1},\partial\varOmega)\,|a'(z_{1})|}
{1-|a(z_{1})|^2}.
\]
\end{lemma}
\begin{proof}
Let us fix $z_{1}\in\varOmega$ and denote by $\gamma=d(z_{1},\partial\varOmega)$ the distance of $z_{1}$ to the boundary of $\varOmega$. 
Then, we choose the matrix $A$ and the function $f$ defined by
\[
A=\begin{pmatrix}z_{1} &2\gamma\\0 &z_{1}\end{pmatrix}, \quad 
f(z)=\frac{a(z) -a(z_{1})}{1-\overline{a(z_{1})}a(z)},\quad\text{so that  }
f(A)=2\gamma \,f'(z_{1})
\begin{pmatrix}0 &1\\0 &0\end{pmatrix}.
\]
The numerical range $W(A)$ is the disk of radius $\gamma$ centered at $z_{1}$; thus $W(A)\subset \overline\varOmega$. Furthermore, the function $f$ (which is a conformal mapping from $\varOmega$ onto $\D$) is bounded by 1 on $\varOmega$. Therefore,
\[
C(\varOmega,2)\geq\|f(A)\|=2\,\gamma\,|f'(z_{1})|=2\,\frac{d(z_{1},\partial\varOmega)\,|a'(z_{1})|}
{1-|a(z_{1})|^2}.
\]
\end{proof}

\begin{corollary}
 If $\varOmega$ is an n-sided regular polygon, then 
 $C(\varOmega,2)\!\geq \!2\int_0^1\frac{dt}{(1+t^n)^{2/n}}$.
\end{corollary}
\begin{proof}
We can assume that $\varOmega$ is the image of the unit disk by the conformal map
$g(z)=\int_0 ^z \frac{dt}{(1+t^n)^{2/n}}$. Then, we choose $z_1=0$ and for $a$ the reciprocal function of $g$. So, $a(0)=0$, $a'(0)=1$, $d(0,\partial \varOmega)=\int_0^1\frac{dt}{(1+t^n)^{2/n}}$.

This gives $C(\varOmega,2)\geq 1.7666$ for the equilateral triangle, $C(\varOmega,2)\geq 1.854$ for the square, $C(\varOmega,2)\geq 1.9003$ for the regular pentagon, $C(\varOmega,2)\geq 1.9276$ for the regular hexagon etc.
\end{proof}

{\sc Remark.} 
\emph{Using the matrix $A$ and the function $f$ defined by
\[
A=\begin{pmatrix}0 &2r\\0 &0\end{pmatrix}, \quad f(z)=z, 
\]
we obtain the rough lower bound
\[
C(\varOmega,2)\geq \frac{2\,r}{R},
\]
as soon as the domain $\varOmega$ contains the ball of radius $r$ centered at the origin and is contained in the ball of radius $R$ centered at the origin.}\medskip

\section{ Realization of $C(\varOmega,d)$}\label{atteinte}
Recall that
\[
C(\varOmega,d)=\max\{\Big\|\prod_{j=1}^{r} b_{z_j}(A)\Big\|\,; A\in\C^{d,d},\ W(A)\subset \varOmega,\ z_1,\dots,z_r\in\varOmega,\ r\leq d-1\},
\]
where 
 $b_{z_j}(z)=(a(z){-}a(z_j)/(1{-}\overline{a(z_j)}a(z))$ and $a$ is a conformal mapping from
$\varOmega$ onto $\D$.
\begin{theorem}
\label{th4}
Let $\varOmega $ be a bounded convex domain of the complex
plane without corner. We assume that the extension of the conformal map $a$  to $\overline\varOmega$ belongs to $C^2(\overline\varOmega)$.
 Then there exists a matrix $A\in\C^{d,d}$ satisfying $W(A)\subset\overline\varOmega$ and a Blaschke product
$f$ such that
\[
C(\varOmega,d) = \Vert f(A)\Vert .
\]
\end{theorem}
\begin{proof}
There exists a sequence of Blaschke products
$f_n=\prod_{j=1}^{r_n}b_{z_j^n}$ and of matrices $A_n\in \C^{d,d}$ satisfying
$$
W(A_n)\subset \overline \varOmega,\ \sigma (A_n)\subset \varOmega,\
 \quad \text{and }\lim_{n\to\infty }
\Vert f_n(A_n)\Vert =C(\varOmega,d).
$$
The sequence of $\{z
_j^n\in\varOmega \}$ is clearly bounded.
The sequence of $\{A_n\}$ is also bounded since
$\Vert A_n\Vert \leq 2\,\sup\,\{ |z|\,; z\in W(A_n)\}$.
Therefore (after extraction of a subsequence) we can assume that $A_n\to A$,
$r_n=r$, $ z_j^n\to z_j$\,; we can also assume that, for all $j$, $\lim_{n\to\infty}\|b_{z _j^n}(A_n)\|>1$ (otherwise we could suppress such an index $j$). We set
$f(z)=\prod_{j=1}^{r}b_{z_j}(z)$;
we have $z_j\in \overline \varOmega $ and,
if
$z_j\in \partial  \varOmega $, then the
conforming map satisfies $|a(z
_j)|=1$, and thus $b_{z_j}(z)=-a(z_j )$ is
constant in
$\varOmega$. We also have $\sigma(A)\subset \overline\varOmega $ and
$W(A)\subset\overline\varOmega $.
Note that $a(A_n)\to a(A)$\,; indeed, for every $\varepsilon >0$, there exists a polynomial 
$p$ such that $\|a{-}p\|_{L^\infty(\varOmega)}\leq \varepsilon $, whence
\begin{align*}
\|a(A_n){-} a(A)\|&\leq \|p(A_n){-} p(A)\|+\|(a{-}p)(A_n)\|+\|(a{-}p)(A)\|\\
&\leq  \|p(A_n){-} p(A)\|+2\,\cal Q\,\varepsilon,
\end{align*}
and clearly $\lim \|p(A_n){-} p(A)\|=0$. 
Thus, if for all $j=1,\dots,r$, $1/\overline{a(z _j)}\notin \sigma(a(A))$, 
\[
b_{z _j^n}(A_n)=(a(z _j^n){-}a(A_n))(1{-}\overline{a(z
_j^n)}\,a(A_n))^{-1}
\to  b_{z _j}(A );
\]
this yields $ f_n(A_n)\to f(A)$
\smallskip
and consequently $\Vert f(A)\Vert =C(\varOmega,d)$. 

It remains to consider the opposite case\,:\smallskip

``there exists $ j$ such that $1/\overline{a(z _j)}\in \sigma(a(A))$".

\medskip

\noindent 
If  $1/\overline{a(z _j)}\in \sigma(a(A))$, then necessarily $|a(z _j)|=1$ since 
$\sigma(a(A))\subset\overline\D$ and $a(z _j)\in\overline\D$. This implies $z _j \in\partial \varOmega$, and then from the next lemma $\limsup_{n\to\infty}\|b_{z_j^n}(A_n)\|\leq 1$.
This contradicts previous assumptions and proves the impossibility of this case.
\end{proof}\medskip 

\begin{lemma}
Under the assumptions of Theorem\,\ref{th4}, if $z_n\in\varOmega\to \zeta  \in \partial \varOmega$ and if 
$W(A_n)\subset\overline\varOmega$, then $\limsup_{n\to\infty}\|b_{z_n}(A_n)\|\leq 1$.
\end{lemma}
\begin{proof}
We use induction on the dimension $d$. The result is clear if $d=1$, since
for all $z\in \overline\D$ and $c  \in\D$, $\big|\frac{z-c  }{1-\bar c  z}\big|\leq 1$.

Assume that the result holds up to the dimension $d{-}1$. 
We argue ad absurdum. After extraction of a subsequence if needed, we assume that $A_n\to A$ and $\lim_{n\to\infty} \|b_{z_n}(A_n)\|\geq  1{+}\varepsilon >1$. Then, if $a(\zeta )=1/\overline{a(\zeta)}\notin \sigma(a(A))$,
it follows
\begin{align*}
b_{z _n}(A_n)&=(a(A_n){-}a(z_n))(1{-}\overline{a(z
_n)}\,a(A_n))^{-1}\\
&\hskip2cm\to (a(A ){-}a(\zeta ))(1{-}\overline{a(\zeta)}\,a(A))^{-1}=-a(\zeta ).
\end{align*}
This leads to the contradiction $\lim_{n\to\infty} \|b_{z_n}(A_n)\|=1$.
Thus, we need to have $a(\zeta )\in \sigma(a(A))$, that means $\zeta \in \sigma(A)$.
Let us denote by $k$ the multiplicity of this eigenvalue $\zeta$. We can assume that the
matrices
$A_n$ (and thus $A$) were chosen upper triangular and also that
$a_{11}=\dots=a_{kk}=\zeta $. Then we remark that, for all
$i=1,\dots,k$, and for all $ j>i$, we have $a_{ij}=0$. Indeed (using the
canonical basis $\{e_j\}$ in $\C^d$) the condition
$W(A)\subset \overline \varOmega $ implies
\[
\forall \mu\in \C,\qquad\zeta +\mu a_{ij}+|\mu |^2\,a_{jj}=
(e_i\!+\! \mu e_j)^*A(e_i\!+\! \mu e_j)\in (1+|\mu |^2)\overline \varOmega ,
\]
which is only compatible with $\zeta  \in \partial \varOmega$ in the case
$a_{ij}=
0$. Now, if $k<d$, we write the matrices in block form
\[
A_n= \begin{pmatrix} T_n &C_n\\ 0 &S_n \end{pmatrix} , \quad
A= \begin{pmatrix} \zeta  I_k &0\\ 0 &S \end{pmatrix}.
\]
For $n$ large enough $\sigma(T_n)\cap \sigma(S_n)=\emptyset$; therefore, we can
define
$M_n$ as the unique solution of
\[
M_nS_n-T_nM_n=C_n,
\]
which we can also write as
\[
M_n=(C_n+(T_n\!-\! \zeta  I_k)M_n)(S_n\!-\! \zeta  I_{d-k})^{-1}.
\]
Since $C_n\to 0$ and $T_n\!-\! \zeta  I_k\to 0 $,
we deduce
$M_n \to 0$.  We remark now that
\[
A_n=\begin{pmatrix} I &-M_n\\ 0 &I \end{pmatrix}
\begin{pmatrix} T_n &0\\ 0 &S_n \end{pmatrix}
\begin{pmatrix} I &M_n\\ 0 &I \end{pmatrix};
\]
thus
\[
b_{z _n}(A_n)=\begin{pmatrix} I &-M_n\\ 0 &I \end{pmatrix}
\begin{pmatrix} b_{z _n}(T_n) &0\\ 0 &b_{z _n}(S_n) \end{pmatrix}
\begin{pmatrix} I &M_n\\ 0 &I \end{pmatrix}.
\]
This implies
$\Vert b_{z _n}(A_n)\Vert \leq (1+\Vert M_n\Vert)^2\max\{\|b_{z _n}(T_n)\|,\|b_{z _n}(S_n)\|\}$. Since $W(T_n)\subset W(A_n)$ and $W(S_n)\subset W(A_n)$, it follows from the induction hypothesis that 
\[
\limsup_{n\to\infty}\|b_{z_n}(A_n)\| \leq \max(\limsup_{n\to\infty}\|b_{z_n}(T_n)\|,\limsup_{n\to\infty}\|b_{z_n}(S_n)\|)\leq 1,
\]
 which leads to a contradiction.
Therefore the only remaining possible case
is $k=d$.

It remains to consider this case $A=\zeta  $, $\zeta   \in \partial \varOmega\cap \sigma(A)$,  and to show that it is impossible. Let us introduce now $x_n\in\partial \varOmega$, a point such that $|x_n{-}z_n|=\min_{x\in \partial \varOmega}|x{-}z_n|$, $e^{i\theta _n}$ the unit inward normal to $\partial \varOmega$ at $x_n$, and set $\alpha_n=|z_n{-}x_n|{+}\|A_n{-}x_n\|$. Note that $\alpha_{n}\to 0$ as $n\to\infty$. We write
$z_n=x_n+\alpha_ne^{i\theta _n}\xi_n$,  $A_n=x_n+\alpha_ne^{i\theta _n}B_n$, whence $\xi _n>0$, $\xi _n+\|B_n\|=1$ and
${\rm Re}\, \langle B_n v,v\rangle\geq 0$ for all $v\in \C^d$, since $W(A_n)\subset\overline\varOmega$.
After a new extraction of a subsequence, we may assume that $\xi _n\to\xi \geq 0$ and $B_n\to B$, 
and then have ${\rm Re}\, \langle B v,v\rangle\geq 0$ for all $v\in \C^d$ and $\xi +\|B\|=1$.

From the smoothness of the conformal map, $a\in C^2(\overline\varOmega)$, we deduce\smallskip

$\min_{x\in \partial \varOmega}|a'(x)|>0$, $\overline{a(x_n)}a'(x_n) e^{i\theta _n}=a(x_n)\overline{a'(x_n) e^{i\theta _n}}$, $|a(x_n)|=1$,
\begin{align*}
a(A_n){-}a(z_n)&=\alpha_n a'(x_n)e^{i\theta _n} (B_n{-}\xi_n)+{\rm o}(\alpha_n),\\
1{-}\overline{a(z_n)}a(A_n)&=1-\overline{(a(x_n){+}\alpha_n a'(x_n)e^{i\theta _n} \xi _n})\\
&\hskip3cm(a(x_n){+}\alpha_n a'(x_n)e^{i\theta _n} B_n)+{\rm o}(\alpha_n)\\
&=-\alpha_n a(x_n)\overline{a'(x_n)e^{i\theta _n} }\xi _n
-\alpha_n \overline{a(x_n)}a'(x_n)e^{i\theta _n} B_n+{\rm o}(\alpha_n)\\
&=-\alpha_n \overline{a(x_n)}a'(x_n)e^{i\theta _n} (B_n+\xi _n)+{\rm o}(\alpha_n).
\end{align*}
If $B{+}\xi $ is invertible, then $b_{z_n}(A_n)\to-a(\zeta ) (B{-}\xi )(B{+}\xi )^{-1}$.
This situation is impossible, since 
${\rm Re}\, \langle B v,v\rangle\geq 0$ and $\xi \geq 0$ imply $\|(B{-}\xi )(B{+}\xi )^{-1}\|\leq 1$.

Hence, we have to consider the case $B{+}\xi $ is not invertible; then $\xi =0$ and $0$ is an eigenvalue of $B$, of multiplicity $k$.
Arguing as before, the matrices $B_n$ and $B$ may be written in block form as
\[
B_n= \begin{pmatrix} T'_n &C'_n\\ 0 &S'_n \end{pmatrix} , \quad
B= \begin{pmatrix}0 &0\\ 0 &S' \end{pmatrix}.
\]
But now, we cannot have $k=d$ since $\xi +\|B\|=1$. There exists a sequence of matrices $M_n\to 0$ such that
\[
B_n=\begin{pmatrix} I &-M_n\\ 0 &I \end{pmatrix}
\begin{pmatrix} T_n' &0\\ 0 &S'_n \end{pmatrix}
\begin{pmatrix} I &M_n\\ 0 &I \end{pmatrix}.
\]
Therefore, setting $T_n=x_n+\alpha_ne^{i\theta _n}T'_n$, $S_n=x_n+\alpha_ne^{i\theta _n} S'_n$,
we have
\[
A_n=\begin{pmatrix} I &-M_n\\ 0 &I \end{pmatrix}
\begin{pmatrix} T_n &0\\ 0 &S_n \end{pmatrix}
\begin{pmatrix} I &M_n\\ 0 &I \end{pmatrix}.
\]
The contradiction follows, as previously, from the induction hypothesis.
\end{proof}\medskip

{\sc Remark.}{ \it It is only in this last part, that the smoothness of $\partial \varOmega$ occurs.
 In fact, the assumption $a\in C^1(\overline \varOmega)$ with $\min_{x\in \partial \varOmega}|a'(x)|>0$ will be sufficient. Note that the convexity of $\varOmega$ implies $a\in C^{0,\alpha}(\overline \varOmega)$,
for all $0\leq \alpha<1$. Clearly, this Lemma does not hold if the boundary of $\varOmega$ has a corner\,; we do not know whether it holds under the hypothesis that $\varOmega$ is convex with continuous tangent.}

\section{Some personal comments on the numerical range}

We refer to \cite{kip,hojo,gura} for a general discussion on the numerical range. This section is only devoted to a few remarks.\medskip

The numerical range of a matrix is a compact and convex subset of the complex plane. Except in the $2\times 2$ case (where it is an ellipse), its boundary is quite involved. From the convexity we  know that it is the intersection of all tangent half-planes containing it. More precisely, if we write $A=B+i\,C\in \C^{d,d}$, with
$B$ and $C$ self-adjoint, if we set $P_{A}(u,v,w):=$\,det$(uB+vC+wI)$, and if we denote by $w_{m}(u,v)$ the largest root of $P_{A}(u,v,\cdot)=0$ (all roots are real since $B$ and $C$ are self-adjoint), then
\[
W(A)=\{z=x\!+\!iy\,;\  x \cos\alpha+y\sin\alpha+w_{m}(\cos\alpha,\sin\alpha)\leq 0,\text{ for all  }\alpha\in[0,2\pi]\}.
\]
This provides an (exterior) approximation of $W(A)$ by computing a finite number of values of $w_{m}(\cdot,\cdot)$.

The tangential approach for the numerical range is simpler than the Cartesian one. From the previous formula we see that $W(A)$ is a part of the algebraic curve with tangential equation  $P_{A}(u,v,w)=0$. This curve is of class $d$, which means that the polynomial $P_{A}$ is of degree $d$. The Cartesian equation of this curve is generically of degree $\frac{d(d-1)}{2}$, which is the maximal degree given by the Pl\"ucker relations.

An interesting characteristic of the numerical range is its good behaviour with respect to perturbations. If $A$ and $B$ denote two bounded operators on a Hilbert space, the Hausdorff distance $d_{_{H}}(W(A),W(B))$ is bounded by $\|A\,-B\|$. The variational approach is a powerful tool for the analysis of partial differential equations. The assumptions are then generally imposed on the  sesquilinear form $\langle Au,u\rangle$ (for instance in the Lax-Milgram Theorem) and can be translated in terms of localization of the numerical range of an unbounded operator $A$. Furthermore, many numerical approximations (finite element methods, spectral methods, wavelets,...) use approximate sesquilinear forms $\langle A_{h}u_{h},u_{h}\rangle$. The corresponding numerical range $W(A_{h})$ then naturally inherits analogous properties to those of $W(A)$.\medskip

In the applications that I have found (see \cite{crzx1}) one never has a perfect knowledge of the numerical range, but only a localization of the type $W(A)\subset\overline\varOmega$. Therefore, good estimates for the constants $C(\varOmega)$ and $C_{cb}(\varOmega)$ are of great interest.

\section{Supporting arguments for my conjecture}

I have proposed the conjecture $\q=2$ in \cite{crzx}, a little more than ten years ago. Since then, I have tried to prove it, also to find a counter-example, but up to now without success.\medskip

The main argument in favor of my conjecture is a symmetry reason. We have $\q=\sup_{\varOmega}C(\varOmega)$, where $\varOmega$ varies among the non-empty bounded convex sets. Since the constant $C(\varOmega)$ depends only on the shape of $\varOmega$, it seems natural to expect that the upper bound could be attained by a fully symmetric set, i.e., by a disk, and then, in this  case, $C(\varOmega)=2$. Another natural candidate for realizing the upper bound is the very flat case where $\varOmega$ is a strip. For this case, an
empirical extrapolation from my numerical evaluations of $C(S_{0},d)$ with $d=2,4,6,8$ seems to confirm that $C(S_{0})\leq 2$. But the complexity of computations drastically increases with the dimension $d$.\medskip

I have succeeded to show that $\q(2)=2$ \cite{crzx}. I have made many numerical tests for $3\times 3$ matrices and I am convinced that, if $\q(3)$ were larger than 2, I would have succeeded to exhibit a  $3\times 3$ matrix with $\psi(A)>2$. I have particularly explored the neighborhood of matrices $A$ such that $W(A)$ is a disk (which implies $\psi(A)\leq2$) and $\psi(A)=2$, and numerically verified that $\psi(A)$ then corresponds to a local maximum. 

\section{About my numerical tests for the strip and the sector}

For the numerical computation of $C(\varOmega,d)$, it is generally difficult to take into account the constraint $W(A)\subset \varOmega$, but this is quite easy in the strip or sector case. We first consider the strip case $\varOmega=S_{0}:=\{z\in\C\,; |\rm{Im}\,z|<1\}$. It can be seen \cite{crzx} that there exist a matrix $A\in \C^{d,d}$ and a holomorphic function $f$ in $S_{0}$ such that
\[
C(S_{0},d)=\|f(A)\|, \quad\text{with }W(A)\subset \overline{S_{0}}\ {\rm and  }\  |f(z)|\leq1\text{  in }S_{0}.
\]
Furthermore $f$ has the form 
\begin{align*}
f(z)=\prod_{j=1}^{d-1}\frac{\tanh \frac{\pi}{4}z-\zeta_{j}}
{1-\bar\zeta_{j}\tanh \frac{\pi}{4}z} \quad\text{with }|\zeta_{j}|\leq1,\hskip4cm\\
\text{or, equivalently }f(z)=\prod_{j=1}^{d-1}\frac{\exp \frac{\pi}{2}z-\gamma_{j}}
{\exp \frac{\pi}{2}z+\bar\gamma_{j}} \quad\text{with }\rm{Re}\, \gamma_{j}\geq 0.
\end{align*}
Note that the conformal mapping $z\mapsto \tanh \frac{\pi}{4}z$ is one to one from $S_{0}$ onto the unit disk $D$ and  $z\mapsto \exp \frac{\pi}{2}z$  maps the strip $S_{0}$ onto the half-plane ${\rm Re}\, z>0$.

Since $W(A)$ and $\|f(A)\|$ are invariant under a unitary similarity, we can assume that $A=B+i\,C$, with a self-adjoint matrix $B$ and a real diagonal matrix $C$. The constraint $W(A)\subset \overline {S_{0}}$ then becomes
\[
A=B+i\,C,\quad \text{with }B=B^*, \ C={\rm diag\,}(c_{i}),  \  c_{i}\in[-1,1], i=1,\dots,d.
\]
As a matter of fact, we can assume that $c_{i}=\pm1$. Indeed, if  $|c_{k}|<1$ for some $k$, then for all $z$ satisfying $|z|\leq1-|c_{k}|$, $W(A{+}z\,E_{k})\subset\overline {S_{0}}$, where $E_{k}$ denotes the $d\times d$ matrix with the entry 1 in the $(k,k)$ place, and 0 otherwise. Then, it holds $\|f(A+z\,E_{k})\|\leq C(S_{0},d)=\|f(A)\|$. From the maximum principle
applied to the holomorphic function $z\mapsto f(A+z\,E_{k})$, we deduce $\|f(A+z\,E_{k})\|=\|f(A)\|$, for 
$|z|\leq1-|c_{k}|$. In particular, we can replace $c_{k}$ by $1$ if $c_{k}\geq0$, or by $-1$ otherwise, without changing the value of $\|f(A)\|$.

Therefore, it suffices to consider matrices of the form
\begin{equation}\label{red}
A=\begin{pmatrix} D_{1} &E\\*[5pt]E^*&D_{2} \end{pmatrix} + i\begin{pmatrix} I_{k} &0\\*[5pt]0&-I_{d-k}
 \end{pmatrix},\quad\text{with  }1\leq k<d.
\end{equation}
(We do not need to consider the cases $k=0$ or $k=d$, otherwise $A$ would be a normal matrix, which is not compatible with $C(S_{0},d)>1$.) Furthermore, using the invariance through unitary similarity, we can assume that $D_{1}$ and $D_{2}$ are real diagonal matrices, and that the first line and the last row of $E$ are real. Using also the invariance through a horizontal translation we can assume $\rm{Re}($trace$(A))=0$\,; then, changing $A$ to $-A$ and using a block permutation if needed, we can also assume $k\leq d/2$. Finally 
\[
C(S_{0},d)=\max_{1\leq k\leq d/2}\max_{D_{1},D_{2},E,\gamma}\|f(A)\|, \quad\text{with  }
f(z)=\prod_{j=1}^{d-1}\frac{\exp \frac{\pi}{2}z-\gamma_{j}}
{\exp \frac{\pi}{2}z+\bar\gamma_{j}}.
\]
For each value of $k$, we have an optimization problem, with $2k(d{-}k){+}2d{-}2$ real variables,
with $d{-}1$ positivity constraints, $\rm{Re}\, \gamma_{i}\geq0$.
If $d=2$, $k=1$, this optimization problem has $4$ variables, but if $d=4$ and $k=2$, $14$ variables...

For $d=2$, our numerical experiments find again the known value $C(S_{0},2)=1.5876598...$ 

With $d=4$, $k=2$, starting from random initial data, my program converges to $C(S_{0},2)$ in 65\% of cases, towards $1.59400...$ in 27\% of cases, and towards $1.6723401$ in 7\% of cases. I believe that this last value correspond to $C(S_{0},4)$. Then, we remark that the corresponding matrix $A$ is real and has many symmetries
\[
A=\begin{pmatrix} D+iI &E\\&\\E&D-iI\end{pmatrix},\quad\text{with  }
D=\begin{pmatrix} x_{1} &0\\0&-x_{1} \end{pmatrix},\quad
E=\begin{pmatrix} x_{2} &x_{3}\\x_{3}&x_{2} \end{pmatrix},
\]
$x_1=2.3816...$, $x_2=1.388...$, $x_3=1.2523...$, $\gamma_1=8.566...$, $\gamma_2=1$, $\gamma _3=1/\gamma_1$. (If we include these symmetries in our optimization program, it always converges to $1.6723401$.)

An open problem is to prove these symmetries, and to generalize them to larger values of $d
$, which will be useful for decreasing the number of variables in the optimization program and for allowing computations with larger value of $d$. In this way, I have obtained
the values 1.72662... for $d=6$   and 1.764577 for $d=8$.\bigskip

{\it Experiments for the sector} $S_{\alpha}=\{z\in\C\,; z\neq0\ $ and $|\arg z|<\alpha\}$, $0<\alpha<\pi/2$. It is easily verified that the constraint $W(A)\subset\overline{S_{\alpha}}$ is equivalent to writing $A=B(\cos\alpha \,I\!+\!i\sin \alpha\, C)B$, with self-adjoint matrices $B$ and $C$ together with $\|C\|\leq1$.
We know from \cite{crzx} that there exists a matrix $A\in \C^{d,d}$ and a holomorphic function $f$ in $S_{\alpha}$ such that 
\[
C(S_{\alpha},d)=\|f(A)\|, \quad\text{with }W(A)\subset \overline{S_{\alpha}},\quad |f(z)|\leq1\text{  in }S_{\alpha}.
\]
Furthermore, the function $f$ has the form
\begin{align*}
f(z)=\prod_{j=1}^{d-1}\frac{z^s-\gamma_{j}}
{z^s+\bar\gamma_{j}}, \quad\text{with }s=\frac{\pi}{2\alpha}\ \text{and }\rm{Re}\, \gamma_{j}\geq 0.
\end{align*}
As for the strip, we can assume that $C$ is a diagonal matrix with eigenvalues $+1$ or $-1$ and write the matrix $A$ in the form
\begin{equation}\label{red2}
A=\begin{pmatrix} D_{1} &E\\&\\E^*&D_{2} \end{pmatrix}\begin{pmatrix}e^{i\alpha} I_{k} &0\\&\\0&e^{-i\alpha}I_{d-k}
 \end{pmatrix}\begin{pmatrix} D_{1} &E\\&\\E^*&D_{2} \end{pmatrix}
 ,\quad\text{with  }1\leq k\leq d/2.
\end{equation}
 We can also assume that $D_{1}$ and $D_{2}$ are real diagonal matrices. Then, we can use an optimization program based on the formula
 \[
C(S_{\alpha},d)=\max_{1\leq k\leq d/2}\max_{D_{1},D_{2},E,\gamma}\|f(A)\|, \quad\text{with  }
f(z)=\prod_{j=1}^{d-1}\frac{z^s-\gamma_{j}} {z^s+\bar\gamma_{j}}\text{\quad and }s=\frac{\pi}{2\alpha}.
\]

The numerical tests are more delicate than for the strip. For $d=4$, the iterates often go towards a local maximum, or stop with an INF or NAN (mainly for small $\alpha$, i.e., large $s=\frac{\pi}{2\alpha}$, instability due to the computation of $z^s$; moreover, if they converge to a local maximum 
corresponding to $d=2$, some (inactive) $\gamma_{i}$ may tend to 0 or $\infty$). Using a continuation method, I have succeeded to follow a local maximum of $\|f(A)\|$ converging to 1.587... (i.e., $C(S_{\alpha},2)$) and another converging to 1.672É as $\alpha\to 0$. The values are crossing for $\alpha=2\pi/13$.\medskip

From my numerical tests, it seems that $C(S_{\pi/4},4)=C(S_{\pi/4},2)=\sqrt2$. Note that the quarter of plane corresponds to a simple geometry and to a simple conformal mapping $z\mapsto z^2$ from $S_{\pi/4}$ into the half-plane ${\rm Re}\, z>0$. Maybe the conjecture $C(S_{\pi/4},d)=\sqrt2$
is more tractable.

\section{Numerical tests for  $3\times 3$ matrices} 

These tests are based on the formula
\[
\q(3)=\max_{A\in\C^{3,3}} \psi(A).
\]
The difficulty is the computation of $\psi(A)$. \smallskip

{\sc Remark.} \emph{Since $\psi(A)=\psi(U^*AU)$ for unitary $U$ and $\psi(A)=\psi(\lambda A\!+\!\mu I)$ if $\lambda\neq0$, it suffices to consider upper triangular matrices $A$, with null trace and nonnegative off-diagonal elements satisfying $\sum_{j>i}a_{ij}^2=1$. Then the matrix $A$ only depends on 6 real parameters and the interior of $W(A)$ is non-empty.}

\smallskip
Let $a$ be the conformal mapping of the interior of the numerical range onto the unit disk; then $\psi(A)=\psi_{_{\D}}(a(A))$.
We can split the computation of $\psi(A)$ into three steps:

Step 1. Computation of the boundary of $W(A)$. 

Step 2. Computation of $B=a(A)$. 

Step 3. Computation of $\psi_{_{\D}}(B)$.
\medskip

Step 1. For each value of $\theta_{j}=\frac{2j\pi}{2n+1}$, $j=0,1,\dots,2n$, I have computed the point $z_{j}\in\partial W(A)$ with exterior normal $(\cos\theta_{j},\sin\theta_{j})$. It is given by $z_{j}=w_{j}^*Aw_{j}/w_{j}^*w_{j}$, where $ w_j$ is an eigenvector corresponding to the largest eigenvalue of $\cos\theta_{j}M+\sin\theta_{j}N$ (with the notation $A=M+iN$, $M$ and $N$ self-adjoint), see \cite{hojo}. 

{\sc Remark.} \emph{Generically the largest eigenvalue of $\cos\theta\, M+\sin\theta\, N$ is simple for all $\theta$, and the boundary of $W(A)$ is analytic. Another possibility is that there exists one value of $\theta$ such that the largest eigenvalue is double; then $W(A)$ is the convex hull of a cardioid, 
its boundary is $C^1$ and has a straight-line part. The last possibility occurs if all but one off-diagonal element of $A$ vanish; then the numerical range is the convex hull of one point and an ellipse; in this case $\psi (A)\leq 2$ and $\psi (A)=2$ only if the ellipse is a disk and the point belongs to the disk.}\medskip

Step 2. For the computation of $B=a(A)$, we use the finite divided differences of Newton,
\[
B=a(\lambda_{1})I +a[\lambda_{1},\lambda_{2}](A-\lambda_{1}I) +a[\lambda_{1},\lambda_{2},\lambda_{3}](A-\lambda_{1}I)(A-\lambda_{2}I),
\]
where $\lambda_{j}$ denote the eigenvalues of $A$. But for that, we first need to know the conformal mapping $a$\,; since trace$(A)=0$, we can choose it such that $a(0)=0$. Then, we may write $a(z)=z\,\exp(u\!+\!iv)$, with $u(z)$ and  $v(z)$ harmonic real-valued functions.
Note that $u(z)=-\log |z| $ on $\partial W(A)$, which determines $u$ in $W(A)$ in a unique way.\smallskip

Let us consider a representation
$\partial W(A)=\{\sigma(\theta)\,; \theta\in [0,2\pi]\}$ of the boundary.
Then there exists a unique  real-valued $2\pi$-periodic function $q(\cdot)$ such that, for all $z\in W(A)$, 
\[
(u\!+\!iv)(z)=\int_{0}^{2\pi} q(\theta)\log(\sigma(\theta){-}z)\, d\theta\quad{\rm and}\quad
\int_{0}^{2\pi} q(\theta)\, d\theta=0.
\]
To determine the function $q$, we consider the real part of the previous equation at a point on the boundary $z=\sigma(\varphi)\in\partial W(A)$; this gives
\begin{align*}
\int_{0}^{2\pi} q(\theta)\log|\sigma(\theta){-}\sigma(\varphi)|\, d\theta= -\log|\sigma(\varphi)|,\quad
\text{ for all }\varphi\in[0,2\pi[,
\end{align*}
or equivalently
\[
\int_{0}^{2\pi} q(\theta)\log\Big|\frac{\sigma(\theta){-}\sigma(\varphi)}{e^{i\theta}{-}e^{i\varphi}}\Big|\, d\theta +
\int_{0}^{2\pi} \!\!q(\theta)\log |e^{i\theta}{-}e^{i\varphi}|\,  d\theta= -\log|\sigma(\varphi)|,\ \ 
\forall\,\varphi\in[0,2\pi[.
\]
We discretized this equation using the representation  $\sigma(\theta)$ obtained at Step\,1 and approximating $q(\cdot)$ by a trigonometric polynomial $q_{n}(\cdot)$ of degree $n$, and employing a collocation method at the points $\theta_{j}$, $j=0,1,\dots, 2n$ (it is known that an odd number of collocation points is necessary  for such a method), (see, e.g., \cite{ps}). So, we get an approximation $q_{j}=q_{n}(\theta_{j})$ by solving the system
\begin{align*}
\frac{2\pi}{2n+1}
\sum_{j=0}^{2n} q_{j} \log\Big|\frac{\sigma(\theta_{j})-\sigma(\theta_{i})}{e^{i\theta_{j}}-e^{i\theta_{i}}}\Big|+
\int_{0}^{2\pi} q_{n}(\theta)\log |e^{i\theta}-e^{i\theta_{i}}|\, d\theta= -\log|\sigma(\theta_{i})|,\\
\text{ for }i=0,1,\dots, 2n.
\end{align*}
We approximated the first integral by the trapezoidal formula; of course, if $j=i$, we have to replace $\log\Big|\frac{\sigma(\theta_{j})-\sigma(\theta_{i})}{e^{i\theta_{j}}-e^{i\theta_{i}}}\Big|$ by 
$\log |\sigma'(\theta_{i})|$.
 Recall that, for the remaining integral, there holds
\begin{align*}
\int_{0}^{2\pi} q_{n}(\theta)\log |e^{i\theta}-e^{i\theta_{i}}|\, d\theta= -\frac{2\pi}{2n+1}\sum_{j=0}^{2n} c(j\!-\!i) \,q_{j},\\
\text{with  }c(k)=c(-k)=\sum_{j=1}^{n}\frac{\cos j\theta_{k}}{j}.
\end{align*}

This method is very efficient if the boundary is analytic  (with exponential convergence with respect to $n$). Unfortunately, the behavior deteriorates near the non-generic situations of Step\,1.\medskip

Step 3. For the computation of $\psi_{_{\D}}(B)$, I have used an optimization program exploiting the characterization
\[
\psi_{_D}(B):=\sup_{\zeta_{1},\zeta_{2}}\{\|g(B)\|\,; g(z)=
\frac{z-\zeta_1}{1-\bar\zeta_1 z}\,\frac{z-\zeta_2}{1-\bar\zeta_2 z},\quad \ \zeta_{1},\zeta_{2}\in \D\}.
\]
Several random restarts are necessary to approach the global maximum.\medskip

Although there is no guarantee that this program provides the global maximum, it seems accurate and reliable for the computation of $\psi(A)$ if the boundary of the numerical range has a good analytical behavior, and in this case I have always verified that $\psi (A)\leq 2$. But instabilities occur close to the situations with a straight part on the boundary; indeed, in these cases, the representation by $\sigma (\theta )$ built at Step\,1 has discontinuities; and then our choice of using equidistant $\theta_j$ is not compatible with the collocation method of step 2.

For matrices with a straight part on the boundary, a rational parametrization of the boundary of $W(A)$ is known. We have explored the behavior of $\psi (A)$ for such matrices, but only in the case of real entries.
Since $\psi(A)=\psi(U^*AU)$ for unitary $U$ and $\psi(A)=\psi(\lambda A\!+\!\mu I)$ if $\lambda\neq0$, it suffices to consider matrices of the form
\[
A=
\begin{pmatrix}
 0 &a &b\\-a &0 &b\\-b &-b &1
\end{pmatrix},\qquad \text{with  } a\geq 0, \ b\geq 0.
\]
Then, it can be seen that the boundary of the numerical range is composed of the vertical segment
joining $-ia$ to $ia$ and of the part of the cardioid described by
\[
x(t)=\frac{(1-t^2a^2)^2}{(1{-}t^2a^2)^2{+}2t^2b^2(1{+}t^2a^2)},\ \ 
y(t)=\frac{4\,t\,b^2}{(1{-}t^2a^2)^2{+}2t^2b^2(1{+}t^2a^2)},\quad -\tfrac1a<t<\tfrac1a.
\]
We have used this representation in place of Step\,1 to compute the boundary.
In the following table, we display some values of $\psi (A)$ computed for $0<a\leq 1$ and $0<b\leq 1$. 

We also have numerically noticed that $\psi (A)$ is decreasing  with $a$ and $b$ for larger values of these parameters. 
\begin{table}[h]\small
\begin{tabular}{|c|r|r|r|r|r|r|r|r|r|r|}
\hline
$a \backslash b$&.1&.2&.3&.4&.5&.6&.7&.8&.9&1\cr
\hline
.1 &1.330    &1.712    &1.963    &1.988    &1.870    &1.786    &1.677    &1.584    &1.509    &1.448\cr
\hline
.2 &1.320    &1.692    &1.937    &1.994    &1.931    &1.826    &1.717    &1.621    &1.541    &1.475 \cr
\hline
.3 &1.300    &1.595    &1.899    &1.990     &1.959    &1.870   &1.767    &1.670    &1.586    &1.514\cr
\hline
.4 &1.278    &1.597    &1.851    &1.971    &1.974    &1.908    &1.816    &1.720    &1.633    &1.557\cr
\hline
.5 &1.255    &1.546    &1.794    &1.936    &1.970    &1.930   &1.853    &1.764    &1.678    &1.600\cr
\hline
.6 &1.231    &1.496    &1.732    &1.886    &1.946    &1.933    &1.875    &1.797   &1.715    &1.637\cr
\hline
.7 &1.209    &1.446    &1.667    &1.827   &1.907    &1.917    &1.879    &1.815    &1.741    &1.667\cr
\hline
.8 &1.188    &1.400    &1.604    &1.762   &1.855    &1.885    &1.866    &1.818    &1.754   &1.687\cr
\hline
.9 &1.169   &1.358    &1.543    &1.695    &1.796    &1.841    &1.840   &1.807    &1.756    &1.696\cr
\hline
1 &1.152    &1.319    &1.487    &1.630   &1.733     &1.789    &1.803   &1.785   &1.746   &1.696 \cr
\hline
\end{tabular}
\caption{Numerical values of $\psi (A)$.}
\vskip -.2cm
\end{table}
Note that $\psi (A)=1$ if $b=0$ (since then $A$ is a normal matrix) and that $\psi (A)\leq 2$ if $a=0$
(since then $A$ is unitarily similar to a direct sum of a $1\times1$ and a $2\times2$ matrices)\,; furthermore, if $a=0$, $\psi (A)=2$ only if $b=\frac{1}{2\sqrt2}$ ($W(A)$ is then a disk). We have made similar computations for $a$ and $b$ close to these values and never obtained $\psi (A)>2$.
\medskip

Finally, the matrices satisfying $W(A)=\overline\D$ and $\psi (A)=2$ are natural candidates for the realization of $\mathcal Q(3)$; they are characterized in the next section. We have numerically explored around them and were never led to a contradiction to $\psi (A)\leq 2$. This seems to back up that, at least, they correspond to a local maximum of $\psi $.

\section{Matrices with $W(A)=\overline\D$ and $\psi (A)=2$}

From a result of Ando \cite{ando}, it is known that a matrix $A$ satisfies $W(A)\subset\overline\D$ if, and only if, it can be written in the form
\begin{equation}\label{and}
A=2\,\sin B\ U\cos B \quad\text{with $U$ unitary and} \ 0\leq B=B^*\leq \tfrac\pi 2.
\end{equation}
Now, we give a characterization of the equality $W(A)=\overline\D$.
\begin{lemma}
The numerical range of a matrix $A$ is the closed unit disk if, and only if, $A$ can be written in the form \eqref{and} with 
\begin{equation}\label{eq}
\det(U \,\cos B -z \sin B)=0,\quad \text{for all }z\in \C.
\end{equation}
\end{lemma}
\begin{proof} a) Assume that \eqref{and} and \eqref{eq} hold. Then, for every $\theta \in\R$, there exists a unit vector $v$ such that $U\cos B\  v=e^{i\theta }\sin B\ v$. This yields $\|\cos B\ v\|^2=\|\sin B\ v\|^2=(\|\sin B\ v\|^2{+}\|\cos B\ v\|^2)/2=1/2$. Hence,
\[
v^*Av=2v^*\sin B\ U\cos B\ v=2e^{i\theta }\|\sin B\ v\|^2=e^{i\theta }.
\] 
This shows that $W(A)$ contains the unit circle. We also know that $W(A)$ is convex and, from \eqref{and}, that  $W(A)\subset\overline\D$; thus, $W(A)=\overline\D$.\smallskip

b) Assume that $W(A)=\overline\D$. Then, for every $\theta \in\R$, there exists a unit vector $v$ such that $e^{i\theta }=v^*Av=2v^*\sin B\ U\cos B\ v$. This implies
\[
1=2|v^*\sin B\ U\cos B\ v|\leq 2\|\sin B\ v\|\, \|\cos B\ v\|\leq \|\sin B\ v\|^2{+}\|\cos B\ v\|^2=1.
\] 
We see that the inequalities in these estimates hold obviously as equalities, and this leads to $\|\cos B\ v\|=\|\sin B\ v\|=1/\sqrt2$, $U\cos B\ v=\lambda \sin B\ v$, and $\lambda =e^{i\theta }$. Thus, we deduce 
$\det(U\cos B{-}e^{i\theta }\sin B)=0$, for all $\theta \in \R$. Then, \eqref{eq} follows since a non-degenerate polynomial has a finite number of roots.
\end{proof}
 
{\sc Remark.} \emph{ If $W(A)$ is the closed unit disk, then $0$ is an eigenvalue of $A$ with multiplicity at least $2$. One way to see this is to take $z=0$ in \eqref{eq}, as then we get $\det (\cos B)=0$ and, using $z=\infty$, we obtain $\det (\sin B)=0$. This shows that $\dim$\,{\rm Ker} $\sin(2B)\geq 2$, and then $\dim$\,{\rm Ker}$A^2\geq 2$ follows since $A^2=2\,\sin B\,U\sin (2B)\,U\cos B$.}\bigskip

In particular, this remark shows that the numerical range of a $2\times2$ matrix is the closed unit disk if, and only if, the matrix is unitarily similar to $\begin{pmatrix}0 &2\\0 &0\end{pmatrix}$; this result was already known \cite{gura,ckli}. Then $\psi (A)=2$ also holds.

We now turn to the case of $3\times 3$ matrices.
\begin{lemma}
 Let us consider a matrix $A\in \C^{3,3}$. Then, $W(A)=\overline\D$ and $\psi (A)=2$  if, and only if,
 $A$ is unitarily similar to a matrix belonging to one of the two following families
\begin{eqnarray*}
\begin{pmatrix}
0 &0 &2\\0 &\xi &0\\0 &0 &0\end{pmatrix},\quad\hbox{with }\ \xi\in\C,\ |\xi|\leq1,\hskip4cm\\
e^{i\psi} \begin{pmatrix}
0 &\sqrt{2}\cos\varphi &2\sin\varphi  \\0 &-\sin\varphi&\sqrt2 \cos\varphi \\0 &0 &0 
\end{pmatrix},\quad\hbox{with }\ \varphi  \in[0,\tfrac\pi 2],\ \psi\in\R.
\end{eqnarray*}
\end{lemma}
\begin{proof}
We write $A$ in the form \eqref{and}. We can assume that the matrix $B$ is diagonal and, from the previous remark, that 
 \begin{eqnarray*}
\sin B=\diag(1,\sin b,0),\qquad \cos B=\diag(0,\cos b,1), \quad 0\leq b\leq \frac\pi 2.
\end{eqnarray*}
Condition\,\eqref{eq} reads: $z\,(z\,u_{33}\sin b-(u_{22}u_{33}{-}u_{23}u_{32})\cos b)=0$, for all $z$, which implies
\[
``\sin b\ \cos b=0" \quad\text{or}\quad ``u_{33}=0\ \text{and }\ u_{23}u_{32}=0".
\] 

\noindent{\it Case\,1\,: $\sin b\cos b=0$.} Then
\[
A=2 \begin{pmatrix}
0 &u_{12} &u_{13} \\0 &0 &0\\0 &0  &0\end{pmatrix}\quad\hbox{or\quad}
A=2 \begin{pmatrix}
0 &0 &u_{13} \\0 &0 &u_{23}\\0 &0  &0\end{pmatrix}.
\]
These matrices are unitarily similar to the matrix $\begin{pmatrix}
0 &0 &\alpha \\0 &0 &0\\0 &0  &0\end{pmatrix}$
with $\alpha=\|A\|$. Necessarily $\alpha=2$ if $W(A)=\overline\D$ and then it is easily seen that $\psi (A)=2$. This situation corresponds to the first family with $\xi =0$.
\medskip

\noindent{\it Case\,2\,: $u_{23}=u_{33}=0$.} Then, since the matrix $U$ is unitary, $u_{13}=e^{i\theta}$, $u_{11}=u_{12}=0$; thus
\[
A=\begin{pmatrix}
0 &0 &2e^{i\theta } \\0 &u_{22}\sin 2b &0\\0 &0  &0\end{pmatrix}.
\]
This matrix is unitarily similar to the matrix of the first family with $\xi =u_{22}\sin 2b$ and it is easily seen that $\psi (A)=2$.
\medskip

\noindent{\it Case\,3\,: $u_{32}=u_{33}=0$, $u_{23}\neq 0$, $\sin b \cos b\neq 0$.}
Then $|u_{31}|=1$, $u_{11}=u_{21}=0$,
\[
A= \begin{pmatrix}
0 &2u_{12}\cos b &2u_{13} \\0 &u_{22}\sin 2b  &2u_{23}\sin b\\0 &0 &0 \end{pmatrix},
\] 
and $|u_{12}u_{23}-u_{13}u_{22}|=|\det (U)|=1$.
In this case, \eqref{eq} is satisfied; thus $W(A)=\overline\D$. Assume now that $\psi (A)=2$ and set
\[
X=\diag(1,\max(1,2\cos b),2),\quad C=XAX^{-1}=2\,X\sin B\ U\cos BX^{-1}.
\]
Clearly, $\|X\|=2$, $\|X^{-1}\|=1$, $\|X\sin B\|=1$, $\|2\cos B\ X^{-1}\|=1$, whence $\|C\|\leq 1$.
Since $\psi (A)=2$, there exist a two factors Blaschke product $g$ and two unit vectors $u$ and $v$ such that $g(A)u=2v$. Then, since $\|g( C)\|\leq 1$ (von Neumann inequality),
\[
2=v^*g(A)u=v^*X^{-1}g( C)Xu\leq \|X^{-1}v\|\,\|Xu\|\leq  \|X^{-1}\|\,\|v\|\,\|X\|\,\|u\|\leq 2.
\]
This yields $\|Xu\|=\|X\|\,\|u\|$, $\|X^{-1}v\|=\|X^{-1}\|\,\|v\|$, and thus for some $\varphi $ and $\theta \in \R$, $u=e^{i\varphi }e_3$, $v=e^{i(\varphi +\theta)}e_1$. Notice that $u^*A=0$, whence $u^*g(A)=g(0)u^*$, and then
\[
g(0)=g(0)u^*u=u^*g(A)u=2u^*v=0.
\]
This allows us to write $g(z)=z\frac{z-\alpha}{1-\bar\alpha z}$ with $|\alpha|<1$ or $\alpha=1$.
Noting that $v=(I-\bar\alpha A)v$, we deduce from $2v=g(A)u$ that $2e^{i\theta }e_1=(A^2-\alpha A)e_3$, which reads
\[
e^{i\theta}=u_{12}u_{23}\sin 2b-\alpha\, u_{13}\quad\text{and}\quad 0=u_{23} \sin b\,(u_{22}\sin2b-\alpha),
\] 
i.e.,
\[
\alpha=u_{22}\,\sin2b\quad\text{and}\quad e^{i\theta }=\sin2b\,(u_{12}u_{23}-u_{13}u_{22}).
\]
Recall that $|u_{12}u_{23}-u_{13}u_{22}|=1$, whence $\sin 2b=1$ and $b=\pi /4$. This leads to
\[
A= \begin{pmatrix}
0 &\sqrt2\,u_{12} &2u_{13} \\0 &u_{22}  &\sqrt2\,u_{23}\\0 &0 &0 \end{pmatrix}.
\] 
We can write $u_{22}=-e^{i\psi}\sin\varphi$ with $\varphi \in[0,\pi /2]$. Using a diagonal unitary similarity if needed, we may assume that $e^{-i\psi }u_{12}\geq 0$, $e^{-i\psi }u_{23}\geq 0$; then, since $U$ is unitary, $u_{12}=u_{23}=e^{i\psi }\cos\varphi $ and $u_{13}=e^{i\psi }\sin\varphi $. This shows that $A$ belongs to the second family.\smallskip

Conversely, if $A$ belongs to the second family, using $\alpha=-e^{i\psi }\sin \varphi $ we easily get $(A^2{-}\alpha A)e_3=2\,e^{i\psi }e_1$ and then $g(A)e_3=2\,e^{i\psi }e_1$; this shows that $\psi (A)=2$.
\end{proof} 

\section{Conclusion}

We have only described a primary approach of natural questions concerning the different constants introduced in this paper. Many results may be improved, numerous other directions may be explored. Hereafter, in complement with our main conjectures, we list some problems that we have failed to solve.\smallskip

{\it Open problems.}
\begin{enumerate}[\tiny $\bullet$]\itemsep=2pt

\item It will be interesting to construct an efficient method for the computation of $\psi_{_\D}(B)$; see \eqref{psid}. For $2\times 2$ matrices there exists an explicit formula, but even for $3\times3$ matrices, we have only succeeded to use optimization algorithms, without guarantee of convergence towards the global maximum.

\item Similarly we do not know a reliable algorithm for computing the matrix $S$ in the Paulsen characterization  \eqref{psicbd} of $\psi_{_{cb,\D}}(B)$. To my knowledge, the existing proofs of the corresponding theorem use the Hahn-Banach theorem and thus are not constructive.

\item Is the estimate $\psi(A)\leq\psi_{cb}(A)\leq2$ valid, if $A^3=0$? 

\item Are the estimates $\psi(A)\leq\q(3)$ and $\psi_{cb}(A)\leq\q_{cb}(3)$ valid, if $A$ is a cubic matrix (i.e, if $p(A)=0$ for some polynomial of degree 3)?

\item Is it true that  $\ C(\varOmega,d)=C_{cb}(\varOmega,d)$ for any convex domain $\varOmega $\,? 
 
 (It is known from \cite{pal} that $\ C(\varOmega,2 )=C_{cb}(\varOmega,2)$.)

\item  Is it true that $ C_{cb}(S_0)\leq 2 $\,? (We especially mention this case, since then the constraint $W(A)\subset S_{0}$ is quite simple.)

\item Is it possible to get an estimate  of $C_{cb}(S_\alpha )$
from the knowledge of $C_{cb}(S_0)$ and of $C_{cb}(S_{\pi /2})=1$ by some interpolation technique? 
 
\item  Does the condition $ C(\varOmega,d)= 2 $ imply $\varOmega$ is a disk? (This is the case if $d= 2$; see \cite{crzx}.)

\item Let us consider a sequence of matrices $A_d\in \C^{d,d}$ which achieve $C(\Omega,d)$. A natural question, suggested by a referee, is: will $\lim_{d\to \infty} W(A_d)=\overline\Omega$? Note that this is not generally true for finite $d$. Indeed, assume that $\Omega$ is an ellipse. We have seen that $C(\Omega,2)\geq 1.5$; but, if $W(A_2)=\overline\Omega$ and if $A_2$ achieves $C(\Omega,2)$, then $C(\Omega,2)=\psi (A_2)$. The value
of $ \psi (A_2)$ only depends of the eccentricity of $W(A_2)$ (see \cite[Theorem\,2.5]{crzx}) and tends to 1 if the eccentricity tends
to 1, which contradicts $C(\Omega,2)\geq 1.5$.

\item Is $\q=\sup_{\varOmega}\ C(\varOmega )$ (resp.  $\q_{cb}=\sup_{\varOmega}\ C_{cb}(\varOmega )$) attained by some domain $\varOmega\,$? (We have seen in Theorem\,\ref{att} that this is the case for $\q(d)$ and $\q_{cb}(d)$.) Is it attained by a domain  $\varOmega$ which is symmetric with respect to the real axis?

\item  Is $C_{cb}(\varOmega,d )$ attained by some matrix $A$\,?  (In Section\,\ref{atteinte}, we have shown that this is the case for $C(\varOmega,d)$, if the boundary of $\varOmega$ is sufficiently smooth\,; this corrects a flaw in the proof of \cite[Theorem\,3.2]{crzx}.)

\item  Is $ C(\varOmega)$ (resp. $C(\varOmega,d)$) a continuous function of $\varOmega$ (for instance with respect to the  Hausdorff distance)? (We have seen a proof of the lower semi-continuity in Lemma\,\ref{sci}.) At least, is $ C(\varOmega)$ converging to 2 as $ \varOmega$ tends to the unit disk?

\item In the case where the boundary of $\varOmega$ is a branch of a hyperbola with angle $2\alpha$, is the equality $ C(\varOmega,d )=C(S_{\alpha},d)$ valid?

\item In the case of $\varOmega$ symmetric with respect to the real axis, does the value of $ C(\varOmega,d )$ (resp.  $C_{cb}(\varOmega,d)$) change, if in the definition we restrict the matrices $A$ to have real entries? More generally, is it possible to deduce some properties for some matrices
$A$ which realize   $ C(\varOmega,d )$ from the symmetries of $\varOmega$\,?

\item Construct a numerical method for the computation of $C(\varOmega,d)$, $\varOmega$ given, $d=2,3,\dots$ \\(We have only partially succeeded to do this for the strip $S_{0}$ and $d\leq 8$\,;
it is known that $C(S_{0},2)=1.5876598...$ and from our numerical experiments we have been led to guess that  $C(S_{0},4)=1.6723401...$, $C(S_{0},6)= 1.72662... $, $C(S_{0},8)=1.764577...$, but we have no guarantee that our optimization algorithm has converged to a global maximum.)

\item Our numerical experiments suggest that for the quarter of plane $S_{\pi/4}$ it holds that $C(S_{\pi/4},4)=\sqrt2$. Is this true and is it true for all $d\,$?
\end{enumerate}
\medskip

{\it Some more comments.} The numerical range being convex, in this paper we have only considered constants $C(\Omega,d)$,\dots,$C_{cb}(\Omega)$, corresponding to convex domains $\Omega$.
There will be no difficulty to extend their definitions  to non convex subsets of $\C$, but we are not convinced of the usefulness of such extensions. However, convexity is a strong constraint and we could be interested to use non-convex spectral sets. Recall that, a proper subset $X$ of $\C$ is called a $K-$spectral set for the operator $A$ if $X$ contains the spectrum of $A$, and if the inequality
\[
\|r(A)\|\leq K\sup_{z\in X}|r(z)|
\]
holds for all rational functions bounded on $X$; we use the term spectral set when $K=1$.
In this context, inequality\,\eqref{cst} reads {\it $W(A)$ is a $\q$-spectral set for the operator $A$}
while (with an evident corresponding designation) inequality\,\eqref{ccb} means {\it $W(A)$ is a $\qcb$-spectral set for the operator $A$.}
Hereafter we make some suggestions for relaxing the convexity.

If $M^{-1}AM=B_1\oplus\cdots\oplus B_k$, then $X=W(B_1)\cup\cdots\cup W(B_k)$ is a $K$-spectral set for $A$ with $K\leq \max(\psi (B_1),\dots,\psi (B_k))\|M\|\,\|M^{-1}\|\leq \q\|M\|\,\|M^{-1}\|$.\smallskip

As noticed in \cite{chgr}, if $A=\varphi (B)$ for some holomorphic function $\varphi $ and some matrix $B$, then $\varphi (W(B))$ is a $\psi (B)$-spectral set for $A$; note that $\varphi (W(B))$ may be non-convex.\smallskip

We can use simultaneous information on $A$ and on $A^{-1}$. For instance, if we consider the annulus $X_R=\{z\,; R^{-1}\leq |z|\leq R\}$, with $R>1$, then there exists a constant $K( R)$ such that, if $w(A)\leq R $ and $w(A^{-1})\leq R$, then $X_R$ is a $K( R)$-spectral set for $A$, see \cite{crzx2}; recall that $w(A):=\max\{|z|\,; z\in W(A)\}$. (It is not known whether $K( R)$ remains bounded as $R\to1$.) 
For similar directions, see \cite{babecr, bc2}.

\medskip

{\bf Acknowledgdment.} The author is greatly grateful of the valuable comments and suggestions of the two referees as of the editor.

\end{document}